\documentclass[aop,MSNbibl,nameyear,seceqn,dvips]{arximspdf}
\usepackage{mathbh, mathrsfs}

\doi{10.1214/10-AOP559} \volume{39} \issue{2} \pubyear{2011}
\firstpage{587} \lastpage{608}

\makeatletter

\newtheorem{theorem}{Theorem}[section]
\newtheorem{lemma}{Lemma}[section]
\newtheorem{corollary}{Corollary}[section]

\newproclaim{definition}{Definition}[section]
\newproclaim{example}{Example}[section]
\newproclaim{remark}{Remark}[section]

\def\~#1{\tilde{#1}}
\def\%#1{\mathcal{#1}}
\def\eqref#1{(\ref{#1})}
\def\klr#1{(#1)}
\def\bklr#1{(#1)}

\def\bbbkle#1{[#1]}
\def\klg#1{\{#1\}}
\def\bklg#1{\{#1\}}

\def\norm#1{\Vert#1\Vert}
\def\abs#1{\vert#1\vert}
\def\babs#1{\vert#1\vert}

\def\bbbabs#1{\biggl\vert#1\biggr\vert}

\makeatother

\begin{document}
\begin{frontmatter}

\title{New rates for exponential approximation and the theorems of R\'enyi and
Yaglom\thanksref{T1}}
\runtitle{New rates for exponential approximation}
\thankstext{T1}{Supported by the Institute for Mathematical Sciences
at the National University of Singapore.}

\begin{aug}
\author[A]{\fnms{Erol A.}~\snm{Pek\"oz}\ead[label=e1]{pekoz@bu.edu}}
\and
\author[B]{\fnms{Adrian}~\snm{R\"ollin}\corref{}\thanksref{t2}\ead[label=e2]{staar@nus.edu.sg}}
\thankstext{t2}{Supported in part by NUS Grant R-155-000-098-133.}
\runauthor{E. A. Pek\"oz and A. R\"ollin}

\affiliation{Boston University and National University of Singapore}
\address[A]{
School of Management\\
Boston University\\
595 Commonwealth Avenue\\
Boston, Massachusetts 02215\\
\printead{e1}}

\address[B]{Department of Statistics\\
\quad and Applied Probability\\
National University of Singapore\\
6 Science Drive 2\\
Singapore 117546\\
\printead{e2}}
\end{aug}

\received{\smonth{3} \syear{2010}}

\begin{abstract}
We introduce two abstract theorems that reduce a variety of
complex exponential distributional approximation problems to the construction of
couplings.  These are applied to obtain new rates of convergence with respect to
the Wasserstein and Kolmogorov metrics for the theorem of R\'enyi on random
sums and generalizations of it, hitting times for Markov chains, and to obtain a
new rate for the classical theorem of Yaglom on the exponential asymptotic
behavior of a critical Galton--Watson process conditioned on nonextinction.  The
primary tools are an adaptation of Stein's method, Stein couplings, as well as
the equilibrium distributional transformation from renewal theory.
\end{abstract}



\end{frontmatter}

\section{Introduction}

The exponential distribution arises as an asymptotic limit in a wide variety of
settings involving rare events, extremes, waiting times, and quasi-stationary
distributions.  As discussed in the preface of \citet{Aldous1989}, the tremendous
difficulty in obtaining  explicit bounds on the error of the exponential
approximation in more than the most elementary of settings apparently has left a
gap in the literature.  The classical theorem of \citet{Yaglom1947} describing
the asymptotic exponential behavior of a critical Galton--Watson process
conditioned on nonextinction, for example, has a large literature of extensions
and embellishments [see \citet{Lalley2009}, e.g.] but the complex
dependencies between offspring have apparently not previously allowed for
obtaining explicit error bounds. Stein's method, introduced in \citet{Stein1972},
is now a well-established method for obtaining explicit bounds in distributional
approximation problems in settings with dependence  [see \citet{Ross2007} for
an introduction]. Results for the normal and Poisson approximation, in
particular, are extensive but also are currently very actively being further
developed; see, for example, \citet{Chatterjee2008} and
\citet{Chen2009}.

There have been a few attempts to apply Stein's method to exponential
approximation. \citet{Weinberg2005} sketches a few potential applications but
only tackles simple examples thoroughly, and \citet{Bon2006} only considers
geometric convolutions. \citet{Chatterjee2006} breaks new ground by applying the
method to a challenging problem in spectral graph theory using exchangeable
pairs, but the calculations involved are application-specific and far from
elementary. In this article, in contrast, we develop a general framework that
more conveniently reduces a broad variety of complex exponential distributional
approximation problems to the construction of couplings.  We provide evidence
that our approach can be fruitfully applied to nontrivial applications and in
settings with dependence---settings where Stein's method typically is expected
to shine.

The article is organized as follows. In Section~\ref{sec1}, we present two
abstract theorems formulated in terms of couplings. We introduce a
distributional transformation (the ``equilibrium distribution'' from renewal
theory) which has not yet been extensively explored using Stein's method. We
also make use of Stein couplings similar to those introduced in \citet{Chen2009}.
In Section~\ref{sec7}, we give applications using these couplings to obtain
exponential approximation rates for the theorem of R\'enyi on random sums
and hitting times for Markov chains; our approach yields
generalizations of these results not previously available
in the literature. Furthermore, we consider the rate of convergence in the
classical theorem of Yaglom on the exponential asymptotic behavior of a
critical Galton--Watson process conditioned on nonextinction; this is the first
place this latter result has appeared in the literature. In Section~\ref{sec10},
we then give the postponed proofs for the main theorems.

\section{Main results} \label{sec1}

In this section, we present the framework in abstract form that will subsequently
be used in concrete applications in Section~\ref{sec7}.  This
framework is comprised of two approaches that we will describe here and then
prove in Section~\ref{sec10}.

To define the probability metrics used in this article, we need the
sets of test functions
\begin{eqnarray*}
    \%F_{\mathrm{K}} & = & \{\mathrm{I}[ \cdot \leq z] | z\in \mathbh{R}\},
    \\
    \%F_{\mathrm{W}} & =& \{h\dvtx \mathbh{R}\to\mathbh{R} | {}\mbox{$h$ is Lipschitz,
    $\norm{h'}\leq 1$}\},
    \\
    \%F_{\mathrm{BW}} & =& \{h\dvtx \mathbh{R}\to\mathbh{R} | {}\mbox{$h$ is Lipschitz,
    $\norm{h}\leq 1$ and $\norm{h'}\leq 1$}\}
\end{eqnarray*}
and then the
distance between two probability measures $P$ and $Q$ with respect to $\%F$ is
defined as
\begin{equation} \label{1}
    d_{\%F}\klr{P,Q}:=\sup_{f\in\%F}\bbbabs{\int_{\mathbh{R}} f \,d P-\int_{\mathbh{R}} f \,d Q}
\end{equation}
if the corresponding integrals are well-defined. Denote by $\mathop{d_{\mathrm{K}}}$, $\mathop{d_{\mathrm{W}}}$ and
$\mathop{d_{\mathrm{BW}}}$ the respective distances corresponding to the sets
$\%F_{\mathrm{K}}$, $\%F_{\mathrm{W}}$ and $\%F_{\mathrm{BW}}$. The subscripts
respectively denote the Kolmogorov, Wasserstein and bounded Wasserstein
distances. We can use the following two relations:
\begin{equation}\label{2}
    d_{\mathrm{BW}}\leq d_{\mathrm{W}}, \qquad d_{\mathrm{K}}(P,\operatorname{Exp}(1))
     \leq 1.74\sqrt{d_{\mathrm{W}}(P,\operatorname{Exp}(1))}.
\end{equation}
The first relation is clear, as $\%F_{\mathrm{BW}}\subset\%F_{\mathrm{W}}, $ and we
refer to \citet{Gibbs2002} for the second relation. It is worthwhile noting that
the second inequality can yield optimal bounds with respect to the $d_{\mathrm{K}}$
metric. This is in contrast to normal approximation where in fact $d_{\mathrm{W}}$
and~$d_{\mathrm{K}}$ often exhibit the same order of convergence and hence the
corresponding equivalent of \eqref{2} for the normal distribution does not yield
optimal bounds on~$d_{\mathrm{K}}$; cf. Corollary~\ref{cor10}.

Our first approach involves a coupling with the \emph{equilibrium distribution}
from renewal theory, and is related to the zero-bias coupling from
\citet{Goldstein1997} used for normal approximation [see also
\citeauthor{Bon2006} (\citeyear{Bon2006}), Lemma~6,  \citeauthor{Goldstein2005b} (\citeyear{Goldstein2005b,Goldstein2007}) and
\citet{Ghosh2009}].

\begin{definition}
Let $X$ be a nonnegative random variable with finite mean. We say that a random
variable $X^e$ has the \emph{equilibrium distribution w.r.t. $X$} if for all
Lip\-schitz~$f$
\begin{equation} \label{3}
    \mathbh{E} f(X) - f(0) =\mathbh{E} X   \mathbh{E} f'(X^e).
\end{equation}
\end{definition}

It is straightforward that this implies
\begin{equation} \label{4}
    \mathbh{P}(X^e\leq x) = \frac{1}{\mathbh{E} X}\int_0^x\mathbh{P}[X>y]\,dy
\end{equation}
and our first result below
can be thought of as formalizing the notion that when $\mathscr{L}(W)$ and  $\mathscr{L}(W^e)$
are approximately equal then $W$ has approximately an exponential distribution.

\begin{theorem}\label{thm1} Let $W$ be a nonnegative random variable with $\mathbh{E}
W= 1$ and let $W^e$ have the equilibrium distribution w.r.t.\ $W.$ Then, for
any
$\beta>0$,
\begin{equation} \label{5}
    d_{\mathrm{K}}\bklr{\mathscr{L}(W),\operatorname{Exp}(1)}\leq 12\beta + 2\mathbh{P}[\abs{W^e-W}>\beta]
\end{equation}
and
\begin{equation}\label{6}
    d_{\mathrm{K}}\bklr{\mathscr{L}(W^e),\operatorname{Exp}(1)}\leq \beta + \mathbh{P}[\abs{W^e-W}>\beta].
\end{equation}
If in addition $W$ has finite second moment, then
\begin{equation} \label{7}
    d_{\mathrm{W}}\bklr{\mathscr{L}(W),\operatorname{Exp}(1)}\leq 2\mathbh{E}\abs{W^e-W}
\end{equation}
and
\begin{equation}\label{8}
    d_{\mathrm{K}}\bklr{\mathscr{L}(W^e),\operatorname{Exp}(1)}\leq \mathbh{E}\abs{W^e-W};
\end{equation}
bound \eqref{8} also holds for $d_{\mathrm{W}}\bklr{\mathscr{L}(W^e),\operatorname{Exp}(1)}$.
\end{theorem}

Our second approach involves an adaptation of the \emph{linear Stein couplings}
introduced in \citet{Chen2009}.

\begin{definition} A triple $(W,W',G)$ of random variables is called a
\emph{constant Stein
coupling}~if
\begin{equation}\label{9}
    \mathbh{E}\klg{Gf(W')- Gf(W)} = \mathbh{E} f(W)
\end{equation}
for all $f$ with $f(0)=0$ and for which the expectations exist.
\end{definition}

Let
\begin{eqnarray*}
    r_1(\%F)=\mathop{\sup_{f\in\%F, }}_{ f(0)=0}\abs{\mathbh{E}\klg{Gf(W')- Gf(W) - f(W)}}
\end{eqnarray*}
and $r_2=\mathbh{E}\babs{\mathbh{E}^{W''}(GD)-1}$, where here and in the rest of the article
$D := W'-W$. The random variable $W''$ is defined on the same probability space
as $(W,W',G)$ and can be used to simplify the bounds (it is typically chosen so
that $r_2=0$); let $D' := W''-W$. At first reading one may simply set $W'' = W$
(in which case typically $r_2 \neq 0$); we refer to \citet{Chen2009} for a more
detailed discussion of Stein couplings. Our next result applies to general
random variables, but useful bounds can only be expected if they are coupled
together so that $r_1(\%F_{\mathrm{W}})$ is small.

\begin{theorem}\label{thm2} Let $W$, $W'$, $W''$ and $G$ be random variables
with finite first moments such that also $\mathbh{E}\abs{GD}<\infty$ and
$\mathbh{E}\abs{GD'}<\infty$. Then with the above definitions,
\begin{equation} \label{10}
   d_{\mathrm{W}}\bklr{\mathscr{L}(W),\operatorname{Exp}(1)}\leq
   r_1(\%F_{\mathrm{W}}) + r_2 + 2r_3 + 2r_3' + 2r_4 + 2r_4',
\end{equation}
where
\begin{eqnarray*}
    r_3&=&\mathbh{E}\bigl|{GD \mathrm{I}[\abs{D}>1]}\bigr|,
    \qquad r_3'=\mathbh{E}\bigl|{(GD-1) \mathrm{I}[\abs{D'}>1]}\bigr|,
    \\
    r_4 &=&\mathbh{E}\abs{G(D^2\wedge1)},\qquad\hspace*{15pt}
     r_4' =\mathbh{E}\bigl|{(GD-1)(\abs{D'}\wedge1)}\bigr|.
\end{eqnarray*}
The same bound holds for $d_{\mathrm{BW}}$ with $r_1(\%F_{\mathrm{W}})$ replaced by
$r_1(\%F_{\mathrm{BW}})$. Furthermore, for any $\alpha$, $\beta$ and $\beta'$,
\begin{eqnarray} \label{11}
    &&d_{\mathrm{K}}\bklr{\mathscr{L}(W),\operatorname{Exp}(1)}\nonumber
   \\[-8pt]\\[-8pt]
    &&\qquad    \leq 2r_1(\%F_{\mathrm{BW}})
        + 2r_2 + 2r_5 + 2r_5'
        +22(\alpha\beta+1)\beta' + 12\alpha\beta^2,\nonumber
\end{eqnarray}
where
\begin{eqnarray*}
    r_5&=&\mathbh{E}\bigl|GD \mathrm{I}[\abs{G}>\alpha\mbox{ or }\abs{D}>\beta]\bigr|,
    \\
    r'_5&=&\mathbh{E}\bigl|(1-GD) \mathrm{I}[\abs{G}>\alpha
        \mbox{ or }\abs{D}>\beta\mbox{ or }\abs{D'}>\beta']\bigr|.
\end{eqnarray*}
\end{theorem}

\subsection{Couplings}\label{sec2}

In this section, we present a way to construct the equilibrium distribution more
explicitly and also discuss a few constant Stein couplings.

\subsubsection{Equilibrium distribution via size biasing}\label{sec3}

Assume that $\mathbh{E} W = 1$ and let $W^s$ have the size bias distribution of $W$,
that is,
\begin{eqnarray*}
    \mathbh{E}\klg{W f(W)} = \mathbh{E} f(W^s)
\end{eqnarray*}
for all $f$ for which the expectation exist. Then, if $U$ has the uniform
distribution on $[0,1]$ independent of all else, $W^e := UW^s$ has the
equilibrium distribution w.r.t.~$W$. Indeed, for any Lipschitz $f$ with
$f(0)=0$ we have
\begin{eqnarray*}
    \mathbh{E} f(W)  = \mathbh{E} f(W)-f(0)
     = \mathbh{E}\klg{Wf'(UW)}
     = \mathbh{E} f'(U W^s) = \mathbh{E} f'(W^e).
\end{eqnarray*}
We note that this construction was also considered by \citet{GoldsteinPC} and it has
been observed by \citet{Pakes1992} that for a nonnegative random variable $W$
with $\mathbh{E} W <   \infty$, we have that $\mathscr{L}(W) = \mathscr{L}(UW^s)$ if and only if $W$
has exponential distribution.

\subsubsection{Exchangeable pairs}\label{sec4}

Let $(W,W')$ be an exchangeable pair. Assume that
\begin{eqnarray*}
    \mathbh{E}^W(W'-W) = -\lambda + \lambda R \qquad\mbox{on $\{W>0\}$.}
\end{eqnarray*}
Then, if we set $G = (W'-W)/(2\lambda)$, we have $r_1(\%F_{\mathrm{BW}}) \leq
\mathbh{E}\abs{R}$ and $r_1(\%F_{\mathrm{W}}) \leq
\mathbh{E}\abs{RW}$.

This coupling was used by \citet{Chatterjee2006} to obtain an exponential
approximation for the spectrum of the Bernoulli--Laplace Markov chain. In order
to obtain optimal rates, \citet{Chatterjee2006} develop more application specific
theorems than ours.

\subsubsection{Conditional distribution of $W$ given $E^c$}\label{sec5}

Let $E$ be an event and let $p=\mathbh{P}[E]$, where $p$ is small. Assume
that $W'$ and $Y$ are defined on the same probability space and that $\mathscr{L}(W') =
\mathscr{L}(W|E^c)$ and $\mathscr{L}(Y) = \mathscr{L}(W|E)$. Then, for any Lipschitz $f$ with
$f(0)=0$, and with $G = (1-p)/p$,
\begin{eqnarray*}
    &&\mathbh{E}\bklg{G f(W') - G f(W)}
    \\
    && \qquad = \frac{1-p}{p}\mathbh{E} f(W') - \frac{1}{p}\mathbh{E} f(W) + \mathbh{E} f(W)
    \\
    && \qquad = \frac{1-p}{p}\mathbh{E} f(W') - \frac{1-p}{p}\mathbh{E}(f(W)|E^c) -
        \mathbh{E}(f(W)|E)+ \mathbh{E} f(W)
        \\
    && \qquad = \frac{1-p}{p}\mathbh{E} f(W') - \frac{1-p}{p}\mathbh{E} f(W') -
        \mathbh{E} f(Y) + \mathbh{E} f(W)
        \\
    && \qquad = \mathbh{E} f(W) - \mathbh{E}\klg{Y f'(UY)},
\end{eqnarray*}
so that $r_1(\%F_{\mathrm{W}}) \leq \mathbh{E} Y$. This coupling is  used by
\citet{Pekoz1996} for geometric approximation in total variation. The Stein
operator used there is a discrete version of the Stein operator used in this
article. Clearly, one will typically aim for an event $E \supset\{W=0\}$ in
order to have $Y = 0$.

\begin{remark}
The roles of $W$ and $W'$ from the previous coupling can be reversed.
Let $E$ and $p$ be as before. However, assume now that $\mathscr{L}(W) = \mathscr{L}(W'|E^c)$
and
$\mathscr{L}(Y) = \mathscr{L}(W'|E)$. Then, it is again straightforward to see that
$(W,W',-1/p)$ is
a constant Stein coupling.
\end{remark}

\section{Applications} \label{sec7}

\subsection{Random sums}\label{sec8}
A classical result of \citet{Renyi1957} 
states that $\mathscr{L}(p\times\break \sum_{i=1}^N X_i)\rightarrow \operatorname{Exp}(1)$ as $p\rightarrow 0$
when $N$ has the $\operatorname{Ge}(p)$ distribution (independent of all else) and $X_i$ are
i.i.d.~with $\mathbh{E} X_i=1$. There have been some
generalizations [see \citet{Brown1990}, \citet{Kalashnikov1997} and the references
therein].
\citet{Sugakova1995}, in particular, gives uniform error bounds for independent
but nonidentically distributed summands with identical means.  Our next result
can be viewed as generalizing this to dependent summands and to
nongeometric~$N$.
 For a random variable~$X$, we denote by $F_X$ its
distribution function and by $F_X^{-1}$ its generalized inverse. We adopt the
standard convention that $\sum_a^b  = 0$ if~$b<a$.

\begin{theorem}\label{thm3} Let
$X=(X_1,X_2,\dots)$ be a sequence of square integrable, nonnegative random
variables, independent of all else, such that, for all $i\geq 1$,
\begin{equation}\label{12}
    \mathbh{E}(X_i|X_1,\dots,X_{i-1}) = \mu_i<\infty\qquad\mbox{almost surely.}
\end{equation}
Let $N$ be a positive, integer valued random variable with $\mathbh{E} N
<\infty$ and let $M$ be a random variable satisfying
\begin{equation}\label{13}
  P(M=m)=\mu_mP(N\geq m)/\mu,\qquad  m=1,2,\ldots
\end{equation}
with
\begin{eqnarray*}
  \mu= \mathbh{E} \sum_{i=1}^N X_i= \sum_{m\geq 1} \mu_mP(N\geq m).
\end{eqnarray*}
Then, with  $W = \mu^{-1}\sum_{i=1}^N X_i$, we have
\begin{eqnarray}  \label{14}
    d_{\mathrm{W}}\bklr{\mathscr{L}(W),\operatorname{Exp}(1)}
    \leq  2\mu^{-1}\Bigl({\mathbh{E}\abs{X_M-X_M^e}
      + \sup_{i\geq 1}\mu_i   \mathbh{E}\abs{N - M}}\Bigr),
\end{eqnarray}
where each $X_i^e$ is a random variable having the equilibrium distribution
w.r.t.\ $X_i$ given $X_1,\dots,X_{i-1}$. If, in addition, $X_i\leq C$ for all
$i$ and $\abs{N - M}\leq K$, then
\begin{equation} \label{15}
    d_{\mathrm{K}}\bklr{\mathscr{L}(W),\operatorname{Exp}(1)} \leq
      12 \mu^{-1}\Bigl\{{\sup_{i\geq1} \norm{F_{X_i}^{-1}-F_{X_i^e}^{-1}}
            + CK}\Bigr\};
\end{equation}
if $K=0$, the same bound also holds for unbounded $X_i$.
\end{theorem}

\begin{pf}
We first show that $W^e :=
\mu^{-1}\bklr{\sum_{i=1}^{M-1} X_i + X_M^e}$ has the equilibrium distribution
w.r.t.~$W$. For a given Lipschitz $f$, we write $g(m)=f(\mu^{-1}\sum_{i=1}^m X_i)$ and we have
\begin{eqnarray*}
    \mu\mathbh{E}\biggl[{\frac{g(M)}{\mu_M} - \frac{g(M-1)}{\mu_M}}\biggr]
    = \sum_{m\geq 0} P(N\geq m) \bigl( g(m) - g(m-1)\bigr) = \mathbh{E} g(N)
\end{eqnarray*} and, for any integer $m>0$,
\begin{eqnarray*}
    \mathbh{E} f'\Biggl({\mu^{-1}\sum_{i=1}^{m-1} X_i +\mu^{-1}X_m^e}\Biggr)
    =  \frac{\mu}{\mu_m}\mathbh{E}\bbbkle{ g(m)-g(m-1)}
\end{eqnarray*}
[using \eqref{3}, \eqref{12} and the assumptions on $X_i^e$]
 that together give $
    \mathbh{E} f'(W^e)= \mathbh{E} f(W).$
Then using
\begin{equation}\label{16}
    W^e - W = \mu^{-1}\Biggl\{{(X_M^e-X_M)
    +\operatorname{sgn}(M-N)\sum_{i=(M\wedge N)+1}^{N\vee M} X_i}\Biggr\}
\end{equation}
we obtain \eqref{14} from \eqref{7}.  Letting
$\beta = \mu^{-1}\bklg{\sup_{i\geq1} \norm{F_{X_i}^{-1}-F_{X_i^e}^{-1}} +
CK}$, and using
 Strassen's theorem we obtain \eqref{15} from \eqref{5};  the
remark after \eqref{15} follows similarly.
\end{pf}

\begin{remark}\label{rem1} Let $N\sim\operatorname{Ge}(p)$ and assume that
the $\mu_i$ are bounded from above and bounded away from $0$. This implies
in particular that $\mu \asymp 1/p$ as $p\to0$. Using
\begin{equation} \label{17}
  d_{\mathrm{W}}\bklr{\mathscr{L}(N),\mathscr{L}(M)} = \inf_{(N,M)}\mathbh{E}\abs{N-M}
\end{equation}
from \citet{Kantorovic1958} [see also
\citet{Vallender1973}],
where the infimum ranges over all possible couplings of $N$ and $M$,  we can replace $\mathbh{E}\abs{N-M}$ in
\eqref{14} by the left-hand side of \eqref{17}. To bound this quantity note first
that from \eqref{13} we have $\mathbh{E} h(M) = \mathbh{E} \klg{\frac{\mu_{N}}{\mu p}h(N)}$
for every function~$h$ for which the expectations exist. Note also that
$\mathbh{E}(\mu_N) = \mu p$.
Let $h$ now be Lipschitz with Lipschitz constant $1$ and assume without loss of
generality that $h(0) = 0$, so that $\abs{h(N)}\leq N$.
Then
\begin{eqnarray*}
  \abs{\mathbh{E}\klg{h(M) - h(N)}}
    & = & \biggl|{\mathbh{E}\biggl\{{\biggl({\frac{\mu_N}{\mu p}-1}\biggr)h(N)}\biggr\}}\biggr|
    \leq \mathbh{E}\biggl|{\biggl\{{\frac{\mu_N}{\mu p}-1}\biggr\} N}\biggr|
    \\
    & \leq &\frac{\sqrt{\operatorname{Var}(\mu_N)\mathbh{E} N^2}}{\mu p}
    \leq \frac{\sqrt{2\operatorname{Var}(\mu_N)}}{\mu p^2}.
\end{eqnarray*}
Hence, under the assumptions of this remark, $\mu^{-1}\sup_i \mu_i\mathbh{E}\abs{N-M}$
 is at most of order $\operatorname{Var}(\mu_N)$ as $p\to0$.
\end{remark}

Next, we have an immediate corollary by coupling stochastically ordered random variables.

\begin{corollary}
In the setting in Theorem~\ref{thm3},  assume either $N \leq_{\mathrm{st}} M$ or $N
\geq_{\mathrm{st}} M$ holds as well as that the $X_i$ are independent and, for each $i$,
we have  $\mathbh{E} X_i=1$ and either $X_i \leq_{\mathrm{st}} X_i^e$ or $X_i \geq_{\mathrm{st}} X_i^e$.
Then
\begin{equation}\label{23}
    d_{\mathrm{W}}\bklr{\mathscr{L}(W),\operatorname{Exp}(1)}
    \leq  2\mu^{-1}\sup_{i\geq 1}
    \biggl|{{\frac{1}{2}}\mathbh{E} X_i^2 - 1}\biggr| +
    2\biggl|{\frac{\mathbh{E} N^2 }{2\mu^2} +\frac{1}{2\mu}- 1}\biggr|
\end{equation}
and, furthermore, if $N$ has a $\operatorname{Ge}(p)$ distribution then
\begin{eqnarray}
\label{27}
    &d_{\mathrm{K}}\bklr{\mathscr{L}(W),\operatorname{Exp}(1)}
    \leq  2.47\biggl(p \displaystyle\sup_{i\geq 1} \biggl|\dfrac{1}{2} \mathbh{E} X_i^2
        - 1\biggr|\biggr)^{1/2}.
\end{eqnarray}
\end{corollary}

\begin{remark}
  \label{20}
A nonnegative random variable $X$ with finite mean is said to be NBUE (new
better than used in expectation) if $X^e \leq_{\mathrm{st}} X$ or NWUE (new worse than
used in expectation) if $X^e \geq_{\mathrm{st}} X$ [see \citet{Shaked2007} and
\citet{Sengupta1995} for other sufficient conditions].  A result similar to
\eqref{27} appears as Theorem~6.1 in \citet{Brown1984} with a larger constant,
though \citet{Brown1990} and \citet{Daley1988} subsequently derived significant
improvements.
\end{remark}

\begin{example}[(Geometric convolution of i.i.d. random variables)] Assume that
$N\sim\operatorname{Ge}(p)$ and that $\mathbh{E} X_1 = 1$. Since
$\mathscr{L}(M) = \mathscr{L}(N),$ we can set $M=N$. Denote by $\delta(\%F)$ the
distance between $\mathscr{L}(X_1)$ and $\operatorname{Exp}(1)$ as defined in \eqref{1} with respect to
the set of test functions $\%F$; define $\delta^e(\%F)$ analogously but between
$\mathscr{L}(X)$ and $\mathscr{L}(X^e)$.
In this case, the estimates of Theorem~\ref{thm3} reduce to
\begin{eqnarray}   \label{25}
    d_{\mathrm{W}}\bklr{\mathscr{L}(W),\operatorname{Exp}(1)} & \leq&
        2 p \delta^e(\%F_{\mathrm{W}}),
        \\\label{26}
    d_{\mathrm{K}}\bklr{\mathscr{L}(W),\operatorname{Exp}(1)} & \leq&
        12 p \norm{F_{X_1}^{-1}-F_{X_1^e}^{-1}}
\end{eqnarray}
which can be compared with the (slightly
simplified)
\begin{equation} \label{26b}
    d_{\mathrm{W}}\bklr{\mathscr{L}(W),\operatorname{Exp}(1)} \leq p\delta(\%F_{\mathrm{W}})
      +2p\delta(\%F_2),
\end{equation}
where
\begin{eqnarray*}
    \%F_2 = \bklg{f\in C^{1}(\mathbh{R})\vert f'\in \%F_{\mathrm{W}}},
\end{eqnarray*}
from \citet{Kalashnikov1997}, Theorem~3.1 for $s=2$, page~151.

Noting
  $\delta(\%F_{\mathrm{W}}) \leq 2 \delta^e(\%F_{\mathrm{W}})$ (using the Kantorovich--Rubinstein theorem),
let $Z\sim\operatorname{Exp}(1)$ and let $h$ be a differentiable function with $h(0) = 0$.
Then, recalling that $\mathscr{L}(Z^e)=\mathscr{L}(Z)$, we have from \eqref{3} that $\mathbh{E} h(Z)
= \mathbh{E} h'(Z)$ and, using again \eqref{3} for $X$ and $X^e$,
\begin{equation}\label{26d}
  \mathbh{E} h(X) - \mathbh{E} h(Z) = \mathbh{E} h'(X^e) - \mathbh{E} h'(Z).
\end{equation}
This implies
\begin{equation}\label{26e}
  \delta(\%F_2) = d_{\mathrm{W}}\bklr{\mathscr{L}(X_1^e),\operatorname{Exp}(1)}
\end{equation}
and hence, from \eqref{8}, we have $\delta(\%F_2) \leq
\delta^e(\%F_{\mathrm{W}})$, so that \eqref{26b} gives a bound which is not as good
as \eqref{25} if the bound is to be expressed in terms of
$\delta^e(\%F_{\mathrm{W}})$.

On the other hand, from \eqref{26e} and the triangle inequality,
\begin{eqnarray*}
  \delta^e(\%F_{\mathrm{W}})
    \leq \delta(\%F_{\mathrm{W}})+d_{\mathrm{W}}\bklr{\mathscr{L}(X^e_1),\operatorname{Exp}(1)}
    = \delta(\%F_{\mathrm{W}})+\delta(\%F_2).
\end{eqnarray*}

Hence, although much broader in applicability, our Theorem~\ref{thm3} yields
results
comparable to those in the literature when specialized to the setting of geometric
convolutions.
\end{example}

\begin{theorem}\label{thm4} Let $X=(X_1,X_2,\dots)$ be a sequence of random
variables with $\mathbh{E} X_i = \mu_i$ and $\mathbh{E} X_i^2 <\infty$. Let $N$, $N'$ and
$N''$ be nonnegative, square integrable, integer valued random variables
independent of the sequence
$X$. Assume that
\begin{eqnarray*}\label{28}
   p:=\mathbh{P}[N=0]>0,\qquad \mathscr{L}(N') = \mathscr{L}(N|N>0),
    \qquad  N''\leq N\leq N'.
\end{eqnarray*}
Define $S(k,l) := X_{k+1}+\cdots+X_{l}$ for $k<l$ and $S(k,l) = 0$ for
$k\geq l$. Let $\mu = \mathbh{E} S(0,N)$ and $W = S(0,N)/\mu$. Then,
\begin{eqnarray*}
    &&d_{\mathrm{W}}\bklr{\mathscr{L}(W),\operatorname{Exp}(1)}
    \\
    && \qquad\leq \frac{qs}{p\mu}
    + \frac{4q\mathbh{E}\klg{S(N,N')(1+S(N'',N))}}{p\mu^2}
      + \frac{4\mathbh{E} S(N'',N)}{\mu},
\end{eqnarray*}
where $q=1-p$, $s^2 = \operatorname{Var}\mathbh{E}(S(N,N')|\%F_{N''})$ and $\%F_k :=
\sigma\klr{X_1,\dots,X_k}$.
If, in addition,
\begin{equation}\label{29}
 X_i\leq C,\qquad
 N'-N\leq K_1, \qquad N-N''\leq K_2,
\end{equation}
for positive constants $C$, $K_1$ and $K_2$, then
\begin{equation} \label{30}
    d_{\mathrm{K}}\bklr{\mathscr{L}(W),\operatorname{Exp}(1)}
    \leq \frac{qs}{p\mu} +
        \frac{22 CK_2}{\mu} + \frac{2C^2K_1(11 K_2+6 K_1)}{p\mu^2}.
\end{equation}
\end{theorem}

\begin{pf} We make use of the coupling construction from Section~\ref{sec5}.
Let $E = \klg{N=0}$, let $Y = 0$, let $W' = \mu^{-1}\sum_{i=1}^{N'} X_i$ and
likewise $W'' = \mu^{-1}\sum_{i=1}^{N''} X_i$. Then the conditions of
Section~\ref{sec5} are satisfied with $G = q/p$ and we can apply
Theorem~\ref{thm2}, in particular \eqref{11}. We have $r_1(\%F_{\mathrm{BW}}) = 0$
as proved in Section~\ref{sec5}. Note now that $D = S(N,N')$ and $D' =
S(N'',N)$. Hence, $r_2 =\mathbh{E}\abs{1-q(p\mu)^{-1}\mathbh{E}(S(N,N')|\%F_{N''})}$. As \eqref{9}
implies that $\mathbh{E}(GD) = \mathbh{E} W = 1$, the variance bound of $r_2$ follows. The
$d_{\mathrm{W}}$-bound follows from~\eqref{10}, using the rough estimates
$r_3+r_4\leq\mathbh{E}\abs{GD^2}$ and $r_3'+r_4'\leq 2\mathbh{E} \abs{D'} + 2\mathbh{E}\abs{GDD'}$ as
we assume
bounded second moments. To obtain the $d_{\mathrm{K}}$-bound choose $\alpha = G = q/p$,
$\beta = CK_1/\mu$ and $\beta' = C K_2/\mu$; then $r_5 = r_5' = 0$. Hence,
\eqref{11} yields
\begin{eqnarray*}
    d_{\mathrm{K}}\bklr{\mathscr{L}(W),\operatorname{Exp}(1)}\leq r_2 + 22(\alpha\beta+1)\beta'
        + 12\alpha\beta^2.
\end{eqnarray*}
Plugging in the value for $r_2$ and the constants, the theorem is proved.
\end{pf}

\begin{example}[(Geometric convolution under local dependence)] If
$N+1\sim\operatorname{Ge}(p)$ (that is, $N$ is a geometric distribution starting at $0$) we
can
choose $N' = N+1$, as $\mathscr{L}(N|N>0) = \mathscr{L}(N+1)$ due to
the well-known lack-of-memory property; hence $K_1 = 1$. Assume now there is a
nonnegative integer $m$ such that, for each $i$, $(X_1,\dots,X_{i})$ is
independent of $(X_{i+m+1},X_{i+m+2},\dots)$. We can set $N'' = \max(N - m, 0)$,
hence $s^2 \leq \operatorname{Var} \mu_{N+1}$, where $\mu_i := \mathbh{E} X_i$. Assume also that
$\mu_i\geq \mu_0$ for some $\mu_0>0$, so that
$\mu\geq \mu_0/p $. Hence, Theorem~\ref{thm4} yields
\begin{eqnarray*}
    d_{\mathrm{K}}\bklr{\mathscr{L}(W),\operatorname{Exp}(1)}
    \leq \frac{\sqrt{\operatorname{Var}(\mu_{N+1})}}{\mu_0} +
        \frac{22Cpm}{\mu_0} + \frac{2C^2p(11 m+6)}{\mu_0^2}.
\end{eqnarray*}
Again, convergence is obtained if $\operatorname{Var}(\mu_{N+1})\to0$ as
$p\to0$; cf. Remark~\ref{rem1}. 
\end{example}

\subsection{First passage times} \label{sec432}

Approximately exponential hitting times for Markov chains have been widely
studied; see \citet{Aldous1989}, \citet{Aldous1992} and \citet{Aldous1993a} for
entry points to this literature.   Let $X_0,X_1,\dots$ be a stationary ergodic
Markov chain with a countable state space  $\%X$, transition probability matrix
$P=(P_{i,j})_{i,j\in\%X}$ and stationary distribution $\pi=(\pi_i)_{i\in\%X}$
and let
\begin{eqnarray*}
   \mathscr{L}( T_{\pi,i}) =\mathscr{L}(\inf\{t\geq 0\dvtx X_t=i\}) \qquad \mbox{starting with $\mathscr{L}(X_0)=\pi$}
\end{eqnarray*}
be the
hitting time on state $i$  started according to the
stationary distribution~$\pi$ and let
\begin{eqnarray*}
   \mathscr{L}( T_{i,j})=\mathscr{L}(\inf\{t>0\dvtx X_t=j\}) \qquad \mbox{starting with $X_0=i$}
\end{eqnarray*}
be the hitting time on state $j$ starting from state $i$.
We also say  a stopping time $T_{i,\pi}$ is a  stationary time starting from
state $i$ if $\mathscr{L} (X_{T_{i,\pi}}|   X_0=i ) =\pi.$

\begin{corollary} \label{corr}
With the above definitions, we have
\begin{equation} \label{32b}
\qquad d_{\mathrm{K}}\bklr{\mathscr{L}(\pi_i T_{\pi,i}),\operatorname{Exp}(1)}\leq 2\pi_i
  + \min\bklg{\pi_i\mathbh{E}|T_{\pi,i}-T_{i,i}|,\mathbh{P} (T_{\pi,i}\neq T_{i,i})}.
\end{equation}
\end{corollary}

\begin{pf}
Using a renewal argument to obtain $\mathbh{P} (T_{\pi,i}=k) = \pi_i \mathbh{P} (T_{i,i}> k)$,
 it is then straightforward to see that $\mathscr{L} (T_{i,i}^e)= \mathscr{L} (T_{\pi,i}+U)$
when $U$ is a uniform random variable on $[0,1]$, independent of all else:
with $f(0)=0$ and using  \eqref{3} we have
\begin{eqnarray*}
  \mathbh{E} f'(T_{\pi,i}+U) &=& \mathbh{E} f(T_{\pi,i}+1)-f(T_{\pi,i})
  \\[-2pt]
   &=& \pi_i\sum_k  \mathbh{P}(T_{i,i}> k) \bigl(f(k+1)-f(k)\bigr)
   \\[-2pt]
   &=& \pi_i\sum_k \sum_{j> k}  \mathbh{P}(T_{i,i}=j)\bigl(f(k+1)-f(k)\bigr)
    \\[-2pt]
   &=& \pi_i\sum_j \sum_{0\leq k< j}  \mathbh{P}(T_{i,i}=j)\bigl(f(k+1)-f(k)\bigr)
   \\[-2pt]
   &=& \pi_i\sum_j   \mathbh{P}(T_{i,i}=j)f(j)
   = \pi_i\mathbh{E}  f(T_{i,i}).
\end{eqnarray*}
We then have
\begin{eqnarray}\label{34}
    d_{\mathrm{K}}\bklr{\mathscr{L}(\pi_i T_{\pi,i}),\operatorname{Exp}(1)}
    & \leq& \pi_i + d_{\mathrm{K}} \bigl(\mathscr{L}\bigl(\pi_i
    (T_{\pi,i}+U)\bigr),\operatorname{Exp}(1)\bigr)\nonumber
    \\[-8pt]\\[-8pt]
    & =& \pi_i + d_{\mathrm{K}} (\mathscr{L}(\pi_i
    T_{i,i}^e),\operatorname{Exp}(1)),\nonumber
\end{eqnarray}
where we use $d_{\mathrm{K}} (\mathscr{L}(T_{\pi,i}),\mathscr{L}(T_{\pi,i}+U))\leq \pi_i$ in the first
line and $\mathscr{L} (T_{i,i}^e)= \mathscr{L} (T_{\pi,i}+U)$ in the second line.
We obtain
inequality \eqref{32b} from \eqref{34} and~\eqref{8}, and then using
\[
\mathbh{E}
|T_{\pi,i}+U-T_{i,i}| \leq \mathbh{E} U + \mathbh{E} |T_{\pi,i}-T_{i,i}| \leq 0.5+ \mathbh{E}
|T_{\pi,i}-T_{i,i}|
\]
  along with \eqref{34} and \eqref{6} using
$\beta=\pi_i$ (since $\{\abs{T_{\pi,i}+U-T_{i,i}}>1\}$ implies $\{T_{\pi,i}\neq
T_{i,i}\}$).
\end{pf}

Below, whenever $T_{i,i}$
 and $ T_{i,\pi}$ are used together in an expression it assumed they are both
based on a single copy of the Markov chain.

\begin{corollary}\label{cor}
With the above definitions and $\rho = \mathbh{P}[T_{i,i}<T_{i,\pi}]$,
\begin{eqnarray}  \label{35b}
&&d_{\mathrm{K}}\bklr{\mathscr{L}(\pi_i T_{\pi,i}),\operatorname{Exp}(1)}\nonumber
\\[-8pt]\\[-8pt]
&&\qquad\leq 2\pi_i+\min\Biggl\{{\pi_i \Bigl(\mathbh{E} T_{i,\pi} + \rho\sup_j \mathbh{E}
T_{j,i}\Bigr),
    \sum_{n=1}^\infty \bigl|P_{i,i}^{(n)}-\pi_i\bigr|}\Biggr\}.\nonumber
\end{eqnarray}
\end{corollary}

\begin{pf}
Letting $X_0=i$, $T_{\pi,i} = \inf\{t\geq 0\dvtx X_{T_{i,\pi}+t}=i\}$,
$T_{i,i} = \inf\{t> 0\dvtx X_{t}=i\}$ and $A=\{T_{i,i}<T_{i,\pi}\}$ we have
\begin{eqnarray*}
    \abs{T_{\pi,i}-T_{i,i}}
    \leq (T_{\pi,i}+T_{i,\pi})I_A+T_{i,\pi}I_{A^c}
    \leq T_{i,\pi}+ T_{\pi,i}I_{A}
\end{eqnarray*}
and the first argument in the minimum of \eqref{35b} follows from \eqref{32b} after
noting $\mathbh{E} [T_{\pi,i} |  A]\leq \sup_j \mathbh{E} T_{j,i}$.

For the second argument in the minimum, let  $X_1,X_2,\dots$ be the stationary Markov chain and let
$Y_0,Y_1,\dots$ be a coupled copy of the Markov chain started in state $i$ at
time 0, but let $Y_1, Y_2,\ldots$ be coupled with $X_1,X_2,\dots$ according to
the maximal coupling of \citet{Griffeath1974} so that we have $\mathbh{P}
(X_n=Y_n=i)=\pi_i\wedge P_{i,i}^{(n)}$.  Let $ T_{\pi,i}$ and $T_{i,i}$ be
hitting times respectively defined on these two Markov chains.  Then
\begin{eqnarray*}
  \mathbh{P} (T_{\pi,i}\neq T_{i,i})
     \leq \sum_n \mathbh{P} (X_n=i,Y_n\neq i)+\mathbh{P} (Y_n=i,X_n\neq i)
\end{eqnarray*}
and since
\begin{eqnarray*}
  \mathbh{P} (X_n=i,Y_n\neq i)
    & =& \pi_i - \mathbh{P} (X_n=i,Y_n= i)
    \\
    & =& \pi_i-\pi_i\wedge P_{i,i}^{(n)}
    \\
    & =& [\pi_i - P_{i,i}^{(n)}]^+,
\end{eqnarray*}
and a similar calculation yields $\mathbh{P} (Y_n=i,X_n\neq i) = [ P_{i,i}^{(n)}
-\pi_i]^+$,
and then we obtain \eqref{35b}.
\end{pf}

\begin{example}
With the above definitions and further assuming $X_n$ is an $m$-dependent
Markov chain, we can let $T_{i,\pi}=m$ and we thus have
\begin{eqnarray*}
  &&d_{\mathrm{K}}\bklr{\mathscr{L}(\pi_i T_{\pi,i}),\operatorname{Exp}(1)}
  \\
  &&\qquad\leq 2\pi_i +
    \min\Biggl\{{\pi_i\Bigl(m+\mathbh{P} (T_{i,i}< m)\sup_j \mathbh{E} T_{j,i}\Bigr),
      \sum_{n=1}^{m-1} \bigl|P_{i,i}^{(n)}-\pi_i\bigr|}\Biggr\}.
\end{eqnarray*}
If we consider flipping a biased coin repeatedly, let $T$ be the number of flips
required until the beginning of a given pattern (that cannot overlap with
itself) of heads and tails of
length $k$ first appears as a run. The current run of $k$ flips can be encoded
in the state space of a $k$-dependent Markov chain
and then applying the second result above we
obtain
\begin{eqnarray*}
    d_{\mathrm{K}}\bklr{\mathscr{L}(\pi_i  T_{\pi,i}),\operatorname{Exp}(1)} \leq \pi_i
    (k+1).
\end{eqnarray*}
Using the ``de-clumping'' trick of counting the flips $T$ preceding the first
appearance of tails followed by $k$ heads in row
we have
\begin{eqnarray*}
    d_{\mathrm{K}}\bklr{\mathscr{L}(qp^kT_{\pi,i}),\operatorname{Exp}(1)}\leq
    (k+2)p^k,
\end{eqnarray*}
where $p=1-q$ is the probability of heads.
Similar results are obtained using Poisson and geometric approximations
respectively in \citeauthor{Barbour1992} [(\citeyear{Barbour1992}), page~164] and \citet{Pekoz1996}.
\end{example}

Recall the definitions of NBUE and NWUE from Remark~\ref{20} and, as discussed
in \citet{Aldous1994}, that stationary reversible continuous-time Markov chain
hitting times are NWUE.  The next results are immediate consequences of
Theorem~\ref{thm1} and \eqref{2}. While \eqref{38} appears to be new, inequality
\eqref{37} appears in \citet{Brown1990}, Lemma 2.3. Inequality \eqref{37a} with a
larger constant of 3.119 appears in \citeauthor{Brown1984} [(\citeyear{Brown1984}), Theorem 3.6] for the NBUE
case and in \citeauthor{Brown1984} [(\citeyear{Brown1984}), equation (5.3)] for the NWUE case; this constant
was later improved in both cases to 1.41 for small $\rho$ in \citet{Daley1988}, equation
(1.7).

\begin{corollary}\label{cor10}
If $W$ is either NBUE or NWUE with $\mathbh{E} W=1,$ finite second moment  and
letting $\rho = |{\frac{1}{2}} {\mathbh{E} W^2 } - 1 |$, we have
\begin{eqnarray}\label{37}
    d_{\mathrm{K}}\bklr{\mathscr{L}(W^e),\operatorname{Exp}(1)}&\leq& \rho,
 \\\label{37a}
    d_{\mathrm{K}}\bklr{\mathscr{L}(W),\operatorname{Exp}(1)}&\leq& 2.47 \rho^{1/2}
\end{eqnarray}
and
\begin{equation}\label{38}
    d_{\mathrm{W}}\bklr{\mathscr{L}(W^e),\operatorname{Exp}(1)}\leq \rho,
    \qquad
    d_{\mathrm{W}}\bklr{\mathscr{L}(W),\operatorname{Exp}(1)}\leq 2\rho.
\end{equation}
\end{corollary}

\subsection{Critical Galton--Watson branching process} \label{sec9}
Let $Z_0=1,Z_1,Z_2,\dots$ be a Galton--Watson branching process with
offspring distribution $\nu=\mathscr{L}(Z_1)$. A theorem due to
\citet{Yaglom1947} states that, if $\mathbh{E} Z_1 = 1$ and
$\operatorname{Var} Z_1 = \sigma^2<\infty$, then $\mathscr{L}(
n^{-1}Z_n |Z_n>0)$ converges to an exponential distribution with mean
$\sigma^2/2$. We give a rate of convergence for this asymptotic under
finite third moment of the offspring distribution using the idea from
Section~\ref{sec3}. Though exponential limits in this context are an
active area of research [see, e.g., \citet{Lalley2009}], the
question of rates does not appear to have been previously studied in
the literature. To this end, we make use the of construction from
\citet{Lyons1995}; we refer to that article for more details on
the construction and only present what is needed for our purpose.

\begin{theorem}\label{thm5} For a critical Galton--Watson branching process with
offspring
distribution $\nu=\mathscr{L}(Z_1)$ such that $\mathbh{E} Z_1^3 < \infty$ we have
\begin{eqnarray*}
    d_{\mathrm{W}}\bigl({\mathscr{L}\bigl(2Z_n/(\sigma^2n)|Z_n>0\bigr),\operatorname{Exp}(1)}\bigr)
  = \mathrm{O}\biggl({\frac{\log n}{n}}\biggr).
\end{eqnarray*}
\end{theorem}

\begin{pf}
First, we construct a size-biased branching tree as in \citet{Lyons1995}. We
assume that this tree is labeled and ordered, in the sense that, if $w$ and $v$
are vertices in the tree from the same generation and $w$ is to the left of $v$,
then the offspring of $w$ is to the left of the offspring of $v$, too. Start in
generation $0$ with one vertex $v_0$ and let it have a number of offspring
distributed according to the size-bias distribution of $\nu$. Pick one of the
offspring of $v_0$ uniformly at random and call it $v_1$. To each of the
siblings of $v_1$, attach an independent Galton--Watson branching process with
offspring distribution $\nu$. For $v_1$ proceed as for $v_0$, that is, give it a
size-biased number of offspring, pick one at uniformly at random, call it $v_2$,
attach independent Galton--Watson branching process to the siblings of $v_2$ and
so on. It is clear that this will always give an infinite tree as the ``spine''
$v_0,v_1,v_2,\dots$ of the tree will never die out.

We next need some notation. Denote by $S_n$ the total number of particles in
generation $n$. Denote by $L_n$ and $R_n$, respectively, the number of particles
to the left (exclusive $v_n$) and to the right (inclusive $v_n$), respectively,
of vertex $v_n$, so that $S_n = L_n + R_n$. We can describe these particles in
more detail, according to the generation at which they split off from the spine.
Denote by $S_{n,j}$ the number of particles in generation $n$ that stem from any
of the siblings of $v_j$ (but not $v_j$ itself). Clearly, $S_n = 1 +
\sum_{j=1}^n S_{n,j}$, where the summands are independent. Likewise, let
$L_{n,j}$ and $R_{n,j}$, respectively, be the number of particles in generation
$n$ that stem from the siblings to the left and right, respectively, of $v_j$
(note that $L_{n,n}$ and $R_{n,n}$ are just the number of siblings of $v_n$ to
the left and to the right, respectively). We have the relations $L_n =
\sum_{j=1}^n L_{n,j}$ and $R_n = 1 + \sum_{j=1}^n R_{n,j}$. Note that, for
fixed~$j$, $L_{n,j}$ and $R_{n,j}$ are in general not independent, as they are
linked through the offspring size of $v_{j-1}$.

Now let $R_{n,j}'$ be independent random variables such that
\begin{eqnarray*}
    \mathscr{L}(R'_{n,j}) = \mathscr{L}(R_{n,j} | L_{n,j} = 0)
\end{eqnarray*}
and, with $A_{n,j} = \{ L_{n,j}=0\}$, define
\begin{equation} \label{39}
    R_{n,j}^* = R_{n,j} I_{A_{n,j}} + R_{n,j}' I_{A_{n,j}^c}
    = R_{n,j} + (R_{n,j}' - R_{n,j}) I_{A_{n,j}^c}.
\end{equation}
Define also $R_n^* = 1 + \sum_{j=1}^n R_{n,j}^*$. Let us collect a few
facts which we will then use to give the proof of the theorem:
\begin{enumerate}[(vii)]
    \item[(i)]for any nonnegative random variable $X$ the size-biased
distribution of $\mathscr{L}(X)$ is the same as the size-biased
distribution of $\mathscr{L}(X | X>0)$;

    \item[(ii)] $S_n$ has the size-biased distribution of
        $Z_n$;

    \item[(iii)] given $S_n$, the vertex $v_n$ is uniformly distributed
among the particles of the $n$th generation;

    \item[(iv)] $\mathscr{L}(R_n^*) = \mathscr{L}(Z_n | Z_n > 0);$

    \item[(v)] $\mathbh{E} \klg{R_{n,j}'I_{A_{n,j}^c}}\leq
    \sigma^2\mathbh{P}[A_{n,j}^c];$

    \item[(vi)] $\mathbh{E}\klg{R_{n,j} I_{A_{n,j}^c}} \leq \gamma
            \mathbh{P}[A_{n,j}^c]$, where $\gamma = \mathbh{E} Z_1^3$;

    \item[(vii)] $\mathbh{P}[A_{n,j}^c]\leq \sigma^2\mathbh{P}[Z_{n-j}>0]\leq C(\nu) /
(n-j+1)$ for some absolute constant $C(\nu)$.
\end{enumerate}
Statement (i) is easy to verify, (ii) follows from
\citet{Lyons1995}, equation~(2.2), (iii) follows from
\citet{Lyons1995}, comment after (2.2), (iv) follows from
\citet{Lyons1995}, proof of Theorem C(i). Using independence,
\begin{eqnarray*}
    \mathbh{E} \klg{R'_{n,j} I_{A_{n,j}^c}} = \mathbh{E} R'_{n,j} \mathbh{P}[A_{n,j}^c]
    \leq \sigma^2 \mathbh{P}[A_{n,j}^c],
\end{eqnarray*}
where the second inequality is due to \citet{Lyons1995}, proof of
Theorem~C(i),
which proves (v). If
$X_j$ denotes the number of
siblings of $v_j$, having the size bias distribution of $Z_1$ minus $1$, we have
\begin{eqnarray*}
    \mathbh{E}\klg{R_{n,j} I_{A_{n,j}^c}}
    &\leq& \mathbh{E}\klg{X_j I_{A_{n,j}^c}}
    \leq \sum_{k}k\mathbh{P}[X_j=k,A_{n,j}^c]
    \\
    &\leq& \sum_{k}k\mathbh{P}[X_j=k]\mathbh{P}[A_{n,j}^c|X_j=k]
    \\
    &\leq& \sum_{k}k^2\mathbh{P}[X_j=k]\mathbh{P}[A_{n,j}^c]
    \leq \gamma \mathbh{P}[A_{n,j}^c],
\end{eqnarray*}
hence (vi). Finally,
\begin{eqnarray*}
    \mathbh{P}[A_{n,j}^c] = \mathbh{E}\klg{\mathbh{P}[A_{n,j}^c|X_j]}
    \leq \mathbh{E}\klg{X_j\mathbh{P}[Z_{n-j}>0]} \leq \sigma^2\mathbh{P}[Z_{n-j}>0].
\end{eqnarray*}
Using Kolmogorov's estimate [see \citet{Lyons1995}, Theorem~C(i)],
we have $\lim_{n\to\infty}n\mathbh{P}[Z_n>0] = 2/\sigma^2$, which
implies (vii).

We are now in the position to prove the theorem using \eqref{5} of
Theorem~\ref{thm1}. Let $c=2/\sigma^2$. Due to (iv) we can set $W = c
R_n^*/n$. Due to (i) and (ii), $S_n$ has the size bias distribution of
$R_n^*$. Let $U$ be an independent and uniform random variable on $[0,1]$. Now,
$R_n - U$ is a continuous random variable taking values on $[0,S_n]$ and, due to
(iii), has distribution $\mathscr{L}(U S_n)$; hence we can set $W^e = c (R_n - U) /
n$. It remains to bound $\mathbh{E}\abs{W-W^e}$. From \eqref{39} and using (v)--(vii),
we have
\begin{eqnarray*}
    nc^{-1}\mathbh{E}\abs{W-W^e}
    & \leq& \mathbh{E} U + \mathbh{E}\babs{R_n^* - R_n}
    \leq 1+\sum_{j=1}^n \mathbh{E}\klg{R_{n,j}'I_{A_{n,j}^c} +
            R_{n,j}I_{A_{n,j}^c}}
            \\
    & \leq& 1+ C(\nu)\sum_{j=1}^n\frac{\sigma^2+\gamma}{n-j+1}
     \leq 1+C(\nu)(\sigma^2+\gamma)(1+\log n).
\end{eqnarray*}
Hence, for a possibly different constant $C(\nu)$,
\begin{eqnarray*}
    \mathbh{E}\abs{W-W^e} \leq \frac{C(\nu)\log n}{n}.
\end{eqnarray*}
Plugging this into \eqref{7} yields the final bound.
\end{pf}

\section{Proofs of main results} \label{sec10}

Our  results are based on the Stein operator
\begin{equation} \label{654}
  A f(x) = f'(x)-f(x)
\end{equation}
and the corresponding Stein equation
\begin{equation} \label{40}
    f'(w)-f(w) = h(w) - \mathbh{E} h(Z),\qquad w\geq 0
\end{equation}
previously studied (independently of each
other and, in the case of the first two, independent of the present
work) by \citet{Weinberg2005}, \citet{Bon2006} and \citet{Chatterjee2006}.
It is straightforward that the solution $f$ to~\eqref{40} can be  written as
\begin{equation} \label{41}
    f(w) = - e^w \int_w^\infty\bigl(h(x)-\mathbh{E} h(Z)\bigr) e^{-x} \,
    dx.
\end{equation}

We next need some properties of the
solution~\eqref{41}. Some preliminary results can be found in \citet{Weinberg2005},
\citet{Bon2006}, \citet{Chatterjee2006} and \citet{Daly2008}. We give
self-contained proofs of the following bounds.

\begin{lemma}[(Properties of the solution to the Stein equation)]\label{lem1}
Let $f$ be the solution to~\eqref{40}. If $h$ is bounded, we have
\begin{equation}\label{42}
    \norm{f} \leq \norm{h},\qquad \norm{f'}\leq 2\norm{h}.
\end{equation}
If $h$ is Lipschitz, we have
\begin{equation} \label{43}
    \abs{f(w)}\leq (1+w)\norm{h'}, \qquad
    \norm{f'} \leq \norm{h'},\qquad
    \norm{f''}\leq 2\norm{h'}.
\end{equation}
For any $a>0$ and any $\varepsilon>0$, let
\begin{equation}\label{44}
    h_{a,\varepsilon}(x) := \varepsilon^{-1}\int_0^\varepsilon \mathrm{I}[x+s\leq a]\, ds.
\end{equation}
Define $f_{a,\varepsilon}$ as in \eqref{41} with respect to $h_{a,\varepsilon}$. Define
$h_{a,0}(x) =
\mathrm{I}[x\leq a]$ and $f_{a,0}$ accordingly.
Then, for all $\varepsilon\geq0$,
\begin{eqnarray} \label{45}
    \norm{f_{a,\varepsilon}}&\leq& 1,
    \qquad
    \norm{f'_{a,\varepsilon}}\leq 1,
    \\  \label{45b}
    \abs{f_{a,\varepsilon}(w+t)-f_{a,\varepsilon}(w)}&\leq&1,
    \qquad
    \abs{f'_{a,\varepsilon}(w+t)-f'_{a,\varepsilon}(w)}\leq 1
\end{eqnarray}
and, for all $\varepsilon>0$,
\begin{equation}
   \qquad \abs{f'_{a,\varepsilon}(w+t)-f'_{a,\varepsilon}(w)} \leq (\abs{t}\wedge 1) \label{46}
        +\varepsilon^{-1}\int_{t\wedge 0}^{t\vee 0} \mathrm{I}[a-\varepsilon\leq w+u\leq a]\,d u.
\end{equation}
\end{lemma}

\begin{pf} Write $\~h(w) = h(w) - \mathbh{E} h(Z)$. Assume now that $h$ is bounded.
Then
\begin{eqnarray*}
    \abs{f(w)} \leq e^w\int_w^\infty \abs{\~h(x)}e^{-x}\,dx \leq  \norm{h}.
\end{eqnarray*}
Rearranging \eqref{40} we have $f'(w) = f(w) + \~h(w)$, hence
\begin{eqnarray*}
    \abs{f'(w)} \leq \abs{f(w)} + \abs{\~h(w)}\leq 2\norm{h}.
\end{eqnarray*}
This proves \eqref{42}. Assume now that $h$ is Lipschitz. We can further assume
without loss of generality that $h(0) = 0$ as $f$ will not change under shift;
hence we may assume that $\abs{h(x)}\leq x\norm{h'}$. Thus,
\begin{eqnarray*}
    \abs{f(w)} \leq e^w\int_w^\infty x\norm{h'}e^{-x}\,dx = (1+w)\norm{h'},
\end{eqnarray*}
which  is the first bound of \eqref{43}. Now, differentiate both sides of
\eqref{40} to obtain
\begin{equation}\label{47}
    f''(w) - f'(w) = h'(w),
\end{equation}
hence, analogous to \eqref{41}, we have
\begin{eqnarray*}
    f'(w) = - e^w \int_w^\infty h'(x)e^{-x}\, dx.
\end{eqnarray*}
The same arguments as before lead to the second and third bound of \eqref{43}.

We now look at the properties of $f_{a,\varepsilon}$. It is easy to check that
\begin{equation}\label{47b}
  f_{a,0}(x) = (e^{x-a}\wedge1)-e^{-a}, \qquad
  f_{a,0}'(x) = e^{x-a}\mathrm{I}[x\leq a]
\end{equation}
is the explicit solution to \eqref{47} with respect to $h_{a,0}$. Now,
it is not difficult to see that, for $\varepsilon>0$, we can write
\begin{eqnarray*}
    f_{a,\varepsilon}(x) = \varepsilon^{-1}\int_0^\varepsilon f_{a,0}(x+s) \,d s
\end{eqnarray*}
and this $f_{a,\varepsilon}$ satisfies \eqref{40}. These representations immediately lead
to the bounds \eqref{45} and \eqref{45b} for $\varepsilon\geq 0$ from the explicit formulas
\eqref{47b}. Now let $\varepsilon>0$; observe that, from \eqref{47},
\begin{eqnarray*}
    f'(x+t)-f'(x) = \bigl(f(x+t)-f(x)\bigr)
    + \bigl(h(x+t)-h(x)\bigr).
\end{eqnarray*}
Again from \eqref{47b}, we deduce that $\abs{f_{a,\varepsilon}(x+t)-f_{a,\varepsilon}(x)}\leq
(\abs{t}\wedge1)$, which yields the
first part of the bound~\eqref{46}. For the second part, assume that $t>0$ and
write
\begin{eqnarray*}
  h_{a,\varepsilon}(x+t)-h_{a,\varepsilon}(x) = \int_0^t h'_{a,\varepsilon}(x+s)\,d s
   = -\varepsilon^{-1}\int_0^t \mathrm{I}[a-\varepsilon\leq x+u\leq a]\, d u.
\end{eqnarray*}
Taking the absolute value this gives the second part of the bound \eqref{46} for
$t>0$; a~similar argument yields the same bound for $t<0$.
\end{pf}

The following lemmas are straightforward and hence given without proof.

\begin{lemma}[(Smoothing lemma)]\label{lem2} For any $\varepsilon>0$
\begin{eqnarray*}
    d_{\mathrm{K}}\bklr{\mathscr{L}(W),\mathscr{L}(Z)} \leq \varepsilon
    + \sup_{a>0}\abs{\mathbh{E} h_{a,\varepsilon}(W)- \mathbh{E} h_{a,\varepsilon}(Z)},
\end{eqnarray*}
where $h_{a,\varepsilon}$ are defined as in Lemma~\ref{lem1}.
\end{lemma}

\begin{lemma}[(Concentration inequality)]\label{lem3} For any random variable $V$,
 \begin{eqnarray*}
    \mathbh{P}[a\leq V \leq b]\leq (b-a) + 2d_{\mathrm{K}}\bklr{\mathscr{L}(V),\operatorname{Exp}(1)}.
 \end{eqnarray*}
\end{lemma}

For the rest of the article, write $\kappa = d_{\mathrm{K}}\bklr{\mathscr{L}(W),\operatorname{Exp}(1)}$.

\begin{pf*}{Proof of Theorem~\ref{thm1}}
Let $\Delta := W-W^e$. Define $I_1 := \mathrm{I}[\abs{\Delta}\leq \beta]$; note that $W^e$ may
not have
finite first moment. With $f$ as in \eqref{40} with respect
to \eqref{44}, the quantity $\mathbh{E} f'(W^e)$ is well defined as $\norm{f'}<\infty$,
and we have
\begin{eqnarray*}
    &&\mathbh{E}\klg{f'(W)-f(W)}
    \\
    &&\qquad =
    \mathbh{E}\bigl\{{I_1\bigl(f'(W)-f'(W^e)\bigr)}\bigr\}+\mathbh{E}\bigl\{{(1-I_1)\bigl(f'(W)-f'(W^e)\bigr)}\bigr\}
=:J_1+J_2.
\end{eqnarray*}
Using \eqref{45}, $\abs{J_2}\leq\mathbh{P}[\abs{\Delta}>\beta]$. Now, using \eqref{47} and in
the last step Lemma~\ref{lem3},
\begin{eqnarray*}
    J_1 & = &\mathbh{E} \biggl\{{I_1\int_0^\Delta f''(W+t) \,d t}\biggr\}
    \\
    & =& \mathbh{E} \biggl\{{I_1 \int_0^\Delta \bigl({f'(W+t) - \varepsilon^{-1}\mathrm{I}[a-\varepsilon \leq W+t
        \leq a]}\bigr)\,d t }\biggr\}
        \\
    & \leq& \mathbh{E}\abs{I_1\Delta} +\int_{-\beta}^0 \mathbh{P}[a-\varepsilon\leq W+t\leq a]\,dt
    \leq 2\beta + 2\beta\varepsilon^{-1}\kappa.
\end{eqnarray*}
Similarly,
\begin{eqnarray*}
    J_1  \geq -\mathbh{E}\abs{I_1\Delta} - \int_0^{\beta} \mathbh{P}[a-\varepsilon\leq W+t\leq a]\,dt
    \geq -2\beta - 2\beta\varepsilon^{-1}\kappa,
\end{eqnarray*}
hence $\abs{J_1}\leq 2\beta + 2\beta\varepsilon^{-1}\kappa$.
Using Lemma~\ref{lem2} and choosing $\varepsilon = 4\beta$,
\begin{eqnarray*}
    \kappa \leq \varepsilon + \mathbh{P}[\abs{\Delta}>\beta] + 2\beta + 2\beta\varepsilon^{-1}\kappa
    \leq  \mathbh{P}[\abs{\Delta}>\beta] + 6\beta + 0.5\kappa.
\end{eqnarray*}
Solving for $\kappa$ proves~\eqref{5}.

To obtain \eqref{6}, write
\begin{eqnarray*}
    \mathbh{E}\bklg{f'(W^e)-f(W^e)}
     &=& \mathbh{E}\bklg{f(W)-f(W^e)}
     \\
    & =& \mathbh{E}\bigl\{{I_1\bigl({f(W)-f(W^e)}\bigr)}\bigr\} +
        \mathbh{E}\bigl\{{(1-I_1)\bigl({f(W)-f(W^e)}\bigr)}\bigr\}.
\end{eqnarray*}
Hence, using Taylor's expansion along with the bounds \eqref{45} for $\varepsilon=0$,
\begin{eqnarray*}
    \abs{\mathbh{E}{f'(W^e)-f(W^e)}}\leq \norm{f'}\mathbh{E}\abs{I_1\Delta} +
\mathbh{P}[\abs{\Delta}>\beta] \leq \beta + \mathbh{P}[\abs{\Delta}>\beta],
\end{eqnarray*}
which gives~\eqref{6}.

Assume now in addition that $W$ has finite variance so that $W^e$ has finite
mean. Then
\begin{eqnarray*}
    \abs{\mathbh{E}\klg{f'(W)-f(W)}}
     = \abs{\mathbh{E}\klg{f'(W)-f'(W^e)}}
     \leq \norm{f''}\mathbh{E}\abs{\Delta}.
\end{eqnarray*}
From the bound \eqref{43}, \eqref{7} follows. Also,
\begin{eqnarray*}
    \abs{\mathbh{E}\bklg{f'(W^e)-f(W^e)}} \leq \norm{f'}\mathbh{E}\abs{\Delta}
\end{eqnarray*}
which yields \eqref{8} from \eqref{45} with $\varepsilon= 0$; the
remark after \eqref{8} follows\break from~\eqref{43}.
\end{pf*}

\begin{pf*}{Proof Theorem~\ref{thm2}}
Let $f$ be the solution \eqref{40} to \eqref{41}, hence $f(0) = 0$, and assume that
$f$ is
Lipschitz. From the fundamental theorem of calculus, we have
\begin{eqnarray*}
    f(W') - f(W) = \int_0^D f'(W+t)\, d t.
\end{eqnarray*}
 Multiplying both sides by $G$ and comparing it with the left-hand side of
\eqref{40}, we have
\begin{eqnarray*}
    f'(W) - f(W) & =& Gf(W')-Gf(W)-f(W)
    \\
        &&{} + (1-GD)f'(W'')
        \\
        &&{} + (1-GD)\bigl(f'(W)-f'(W'')\bigr)
        \\
        &&{} - G\int_0^D \bigl(f'(W+t)-f'(W)\bigr)\,d t.
\end{eqnarray*}
Note that we can take expectation component-wise due to the moment assumptions.
Hence,
\begin{eqnarray*}
    \mathbh{E} h(W)-\mathbh{E} h(Z) =  R_1(f) + R_2(f) + R_3(f) -
    R_4(f),
\end{eqnarray*}
where
\begin{eqnarray*}
    R_1(f) &=& \mathbh{E}\bklg{Gf(W')-Gf(W)-f(W)},
    \\
    R_2(f) &=& \mathbh{E}\{{(1-GD)f'(W'')}\},
    \\
    R_3(f) &=& \mathbh{E}\bigl\{{(1-GD)\bigl(f'(W)-f'(W'')\bigr)}\bigr\},
    \\
    R_4(f) &=& \mathbh{E}\biggl\{{G\int_0^D \bigl(f'(W+t)-f'(W)\bigr)\,d t}\biggr\}.
\end{eqnarray*}
Assume now that $h\in\%F_{\mathrm{BW}}$ and $f$ the solution to \eqref{40}. Then
from \eqref{42} and \eqref{43} we obtain $\norm{f}\leq 1$, $\norm{f'}\leq1$ and
$\norm{f''}\leq 2$. Hence, $f\in\%F_{\mathrm{BW}}$,
$\abs{R_1(f)} \leq r_1(\%F_{\mathrm{BW}})$ and $\abs{R_2(f)} \leq r_2$.
Furthermore,
\begin{eqnarray*}
    \abs{R_3(f)} & \leq &\mathbh{E}\bigl|{(1-GD)\bigl(f'(W'')-f'(W)\bigr)}\bigr|
    \\
    & \leq &2\mathbh{E}\klg{\abs{1-GD}\mathrm{I}[\abs{D'}>1]}
    +2\mathbh{E}\klg{\abs{1-GD}(\abs{D'}\wedge 1)}
    \\
    & =& 2r_3' + 2r_4'
\end{eqnarray*}
and
\begin{eqnarray*}
    \abs{R_4(f)} & \leq& \mathbh{E}\biggl|{G\int_0^D\bigl(f'(W+t)-f'(W)\bigr)\,dt}\biggr|
    \\
    &\leq &2\mathbh{E}\bigl|{GD \mathrm{I}[\abs{D}>1]}\bigr|+
    2\mathbh{E}\babs{G(D^2\wedge 1)}
    \\
    & =& 2r_3 + 2r_4.
\end{eqnarray*}
This yields the $d_{\mathrm{BW}}$ results. Now let $h\in\%F_{\mathrm{W}}$ and $f$ the
solution to \eqref{40}. Then,
from \eqref{42} and \eqref{43}, we have $\abs{f(x)}\leq (1+x)$, $\norm{f'}\leq1$
and $\norm{f''}\leq 2$, hence the bounds on $R_2(f)$, $R_3(f)$ and $R_4(f)$
remain, whereas now $f\in\%F_{\mathrm{W}}$ and, thus, $\abs{R_1(f)}\leq
r_1(\%F_{\mathrm{W}})$. This proves the $d_{\mathrm{W}}$ estimate.

Now let $f$ be the solution to \eqref{40} with respect to $h_{a,\varepsilon}$ as
in~\eqref{44}. Then, from \eqref{45}, we have $\norm{f}\leq 1$ and $\norm{f'}\leq 1$,
hence $f\in\%F_{\mathrm{BW}}$, $\abs{R_1(f)}\leq r_1(\%F_{\mathrm{BW}})$
and $\abs{R_2(f)}\leq r_2$.
Let $I_1 = \mathrm{I}[\abs{G}\leq \alpha,\abs{D}\leq \beta',\abs{D'}\leq\beta']$. Write
\begin{eqnarray*}
 R_3(f) &=& \mathbh{E}\bigl\{{(1-I_1)(1-GD)\bigl(f'(W'')-f'(W)\bigr)}\bigr\}
 \\
    &&{} + \mathbh{E}\bigl\{{I_1(1-GD)\bigl(f'(W'')-f'(W)\bigr)}\bigr\} =: J_1 + J_2.
\end{eqnarray*}
Using \eqref{45}, $\abs{J_1}\leq r_5'$ is immediate. Using \eqref{46} and
Lemma~\ref{lem3},
\begin{eqnarray*}
    \abs{J_2} &\leq&\mathbh{E}\bigl|{(GD-1)I_1\bigl(f'(W'')-f'(W)\bigr)}\bigr|
    \\
    & \leq& (\alpha\beta+1)\beta' +
    (\alpha\beta+1)\varepsilon^{-1}\int_{-\beta'}^{\beta'} \mathbh{P}[a-\varepsilon\leq W + u \leq a] \,d u
    \\
    & \leq& (\alpha\beta+1)\beta' +
    (\alpha\beta+1)\varepsilon^{-1}\int_{-\beta'}^{\beta'}(\varepsilon+2\kappa)\, d u
    \\
    & =& 3(\alpha\beta+1)\beta'
    + 4(\alpha\beta+1)\beta'\varepsilon^{-1}\kappa.
\end{eqnarray*}

Similarly, let $I_2 = \mathrm{I}[\abs{G}\leq \alpha,\abs{D}\leq \beta]$ and write
\begin{eqnarray*}
    R_4(f) &=& \mathbh{E}\biggl\{{G(1-I_2)\int_0^D \bigl(f'(W+t)-f'(W)\bigr)\, d t}\biggr\}
    \\
        &&{}  + \mathbh{E}\biggl\{{G I_2\int_0^D \bigl(f'(W+t)-f'(W)\bigr)\, d t}\biggr\} =: J_3+J_4.
\end{eqnarray*}
By \eqref{45}, $\abs{J_3}\leq r_5$. Using again \eqref{46} and
Lemma~\ref{lem3},
\begin{eqnarray*}
    \abs{J_4}&\leq& \mathbh{E}\biggl\{{G I_2\int_{D\wedge 0}^{D\vee
0}\abs{f'(W+t)-f'(W)}\,d t}\biggr\}
\\
    &\leq& \alpha\mathbh{E}\biggl\{{\int_{-\beta}^{\beta}\biggl[{(\abs{t}\wedge1)
    + \varepsilon^{-1}\int_{t\wedge 0}^{t\vee 0}\mathrm{I}[a-\varepsilon\leq W+u\leq a]\,du}\biggr]\,dt}\biggr\}
    \\
    &\leq& \alpha\beta^2 +
    \alpha\varepsilon^{-1}\mathbh{E}\biggl\{{\int_{-\beta}^{\beta}\int_{t\wedge 0}^{t\vee
        0}(\varepsilon+2\kappa)\,du\,d t}\biggr\} = 2\alpha\beta^2 + 2\alpha\beta^2\varepsilon^{-1}\kappa.
\end{eqnarray*}
Using Lemma~\ref{lem2} and collecting the bounds above, we obtain
\begin{eqnarray*}
    \kappa
    &\leq& \varepsilon + r_1(\%F_{\mathrm{BW}}) + r_2 +
                \abs{J_1} + \abs{J_2} + \abs{J_3} + \abs{J_4}\\
    & \leq& \varepsilon + r_1(\%F_{\mathrm{BW}}) + r_2 + r_5 + r_5'
        +3(\alpha\beta+1)\beta' + 2\alpha\beta^2 \\
    &&{} + \bigl(4(\alpha\beta+1)\beta'+2\alpha\beta^2\bigr)\varepsilon^{-1}\kappa
\end{eqnarray*}
so that, setting $\varepsilon = 8(\alpha\beta+1)\beta'+4\alpha\beta^2$,
\begin{eqnarray*}
    \kappa
    \leq \varepsilon + r_1(\%F_{\mathrm{BW}}) + r_2 + r_5 + r_5'
        +11(\alpha\beta+1)\beta' + 6\alpha\beta^2
    + 0.5\kappa.
\end{eqnarray*}
Solving for $\kappa$ yields the final bound.
\end{pf*}

\section*{Acknowledgments}

The authors would like to express gratitude for the gracious hospitality of
Louis Chen and Andrew Barbour during a visit to the National University of
Singapore in January 2009, where a portion of this work was completed.   We also thank the referees for their
helpful comments. We are indebted to Nathan Ross for
many valuable suggestions and for pointing out an error (and its elegant
solution) in an early version of this paper.  We also thank Mark Brown, Fraser
Daly and Larry Goldstein  for inspiring discussions.


\printaddresses

\end{document}